\documentclass[11pt, reqno]{amsart}

\usepackage{amsmath}
\usepackage{amsthm}
\usepackage{amssymb}
\usepackage[mathscr]{eucal}

\newtheorem{thm}{Theorem}[section]

\newtheorem{lemma}[thm]{Lemma}
\newtheorem{cor}[thm]{Corollary}
\theoremstyle{definition}

\newtheorem{remark}[thm]{Remark}

\numberwithin{equation}{section}

\DeclareMathOperator{\kap}{Cap}

\let\epsilon=\varepsilon
\let\phi=\varphi

\def\LL{\Lambda}

\def\ed{\,\mathrm{d}}

\DeclareMathOperator{\im}{Im}

\newcommand{\R}{\mathbb R}
\newcommand{\Ab}{\mathbf{A}}

\title[Schr\"odinger Operators in exterior domains]{Remark on magnetic Schr\"odinger operators in exterior domains}
\author{Ayman Kachmar}
\author{Mikael Persson}
\address[A. Kachmar and M. Persson]{Aarhus University, Departement of
  Mathematical Sciences, 1530 Ny Munkegade, 8000 Aarhus C, Denamrk}
\email[A. Kachmar]{akachmar@imf.au.dk}
\email[M. Persson]{mpersson@imf.au.dk}

\begin{document}

\begin{abstract}
We study the Schr\"odinger operator with a constant magnetic field
in the exterior of a two-dimensional compact domain. Functions in
the domain of the operator are subject to  a boundary condition of
the third type (Robin condition). In addition to the
Landau levels, we obtain that the spectrum of
this operator consists of clusters of eigenvalues around the Landau
levels and that they do accumulate to the Landau levels  from below.
We give a precise asymptotic formula for the rate of accumulation of
eigenvalues in these clusters, which appears to be independent from
the boundary condition.
\end{abstract}

\maketitle

\section{Introduction}

Magnetic Schr\"odinger operators in domains with boundaries appear
in several areas of physics, one can mention the Ginzburg-Landau
theory of superconductors, the theory of Bose-Einstein condensates,
and of course the study of edge states in Quantum mechanics. We
refer the reader to \cite{AfHe, FH, hosm} for details and additional
references on the subject. From the point of view of spectral
theory, the presence of boundaries has an effect similar to that of
perturbing the magnetic Schr\"odinger operator by an electric
potential. In particular, in both cases, the essential spectrum
consists of the Landau levels and the discrete spectrum form
clusters of eigenvalues around the Landau levels. Several papers are
devoted to the study of different aspects of these clusters of
eigenvalues in domains with or without boundaries. In case of
domains with boundaries, one can cite \cite{FK, Fr, HM, Krmp, Kjmp}
for results in the semi-classical context, \cite{P, puro} and the
references therein for the study of accumulation of eigenvalues.

Let us consider a {\bf compact} and {\bf simply connected} domain
$K\subset\R^2$ with a {\bf smooth} $C^\infty$ boundary. Let us
denote by $\Omega=\R^2\setminus K$. Given a function $\gamma\in
C^\infty(\partial\Omega)$ and a positive constant $b$ (the intensity
of the magnetic field), we define the Schr\"odinger operator
$L_{\Omega,b}^\gamma$ with domain $D(L_{\Omega,b}^\gamma)$ as
follows,
\begin{multline}\label{D-L}
D(L_{\Omega,b}^\gamma)=\bigl\{u\in L^2(\Omega)~:~(\nabla-ib\Ab_0)^ju\in
L^2(\Omega),~j=1,2;\\ \nu_\Omega\cdot(\nabla-ib\Ab_0)u+\gamma
u=0\quad{\rm on~}\partial\Omega\bigr\}\,,
\end{multline}
\begin{equation}\label{Op-L}
L_{\Omega,b}^\gamma u=-(\nabla-ib\Ab_0)^2u\quad\forall~u\in
D(L_{\Omega,b}^\gamma)\,.
\end{equation}
Here, $\Ab_0$ is the magnetic potential defined by
\begin{equation}\label{mp-A0}
\Ab_0(x_1,x_2)=\frac12(-x_2,x_1)\quad\forall(x_1,x_2)\in\R^2\,,
\end{equation}
and $\nu_\Omega$ is the unit {\it outward} normal vector of the
boundary $\partial\Omega$.

The operator $L_{\Omega,b}^\gamma$ is actually the Freidrich's
self-adjoint extension in $L^2(\Omega)$ associated with the
semi-bounded quadratic form
\begin{equation}\label{QF-q}
q_{\Omega,b}^\gamma(u)=\int_\Omega|(\nabla-ib\Ab_0)u|^2\ed
x+\int_{\partial\Omega}\gamma|u|^2\ed S\,,
\end{equation}
defined for all functions $u$ in the form domain of
$q_\Omega^\gamma$,
\begin{equation}\label{D-QF-q}
D(q_{\Omega,b}^\gamma)=H^1_{\Ab_0}(\Omega)=
\{u\in L^2(\Omega)~:~(\nabla-ib\Ab_0)u\in
L^2(\Omega)\}\,.
\end{equation}

The result of the present paper is the following.

\begin{thm}\label{thm-KP}
The essential spectrum of the operator $L_{\Omega,b}^\gamma$
consists of the Landau levels,
\begin{equation}\label{eq-thm-1}
\sigma_{\text{ess}}(L_{\Omega,b}^\gamma)=\{\Lambda_n~:~n\in\mathbb
N\}\,,\quad\Lambda_n=(2n-1)b\quad \forall~n\in\mathbb N\,,
\end{equation}
and for all $\varepsilon\in(0,b)$ and $n\in\mathbb N$, the spectrum
of $L_{\Omega,b}^\gamma$ in the interval
$(\Lambda_n,\Lambda_n+\varepsilon)$ is finite.
Moreover, for all $n\in\mathbb N$, denoting by
$\{\ell_j^{(n)}\}_{j\in\mathbb N}$ the increasing sequence of
eigenvalues of $L_{\Omega,b}^\gamma$ in the interval
$(\Lambda_{n-1},\Lambda_n)$, then the following limit
\begin{equation}\label{eq-thm-2}
\lim_{j\to\infty}
\left(j!(\Lambda_n-\ell_j^{(n)})\right)^{1/j}=\frac{b}2\big{(}{\rm Cap}(K)\big{)}^2\,,
\end{equation}
holds provided that $\min_{x\in\partial\Omega}|\gamma(x)|>C_0$  for
a positive geometric constant $C_0=C_0(\Omega)$. Here ${\rm Cap}(K)$
is the logarithmic capacity of the domain $K=\R^2\setminus\Omega$,
and $\Lambda_0$ is set to be $-\infty$ by convention.
\end{thm}

The geometric constant $C_0$ depends only  on the domain $\Omega$
and will be introduced in the proof of Lemma~\ref{lem:RobDirch}
below. The restriction on $\gamma$ being large is technical and its
purpose is to invert some auxiliary boundary operators
(see Lemma~\ref{lem:RobDirch}). We mention also that Theorem~\ref{thm-KP}
was obtained for the Neumann case ($\gamma\equiv0$) by the second
author in \cite{P}.

As immediate corollary of Theorem~\ref{thm-KP}, we get:

\begin{cor}\label{thm-KP-N}
With the notation of Theorem~\ref{thm-KP}, for any function
$\gamma\in L^\infty(\partial\Omega)$, the following asymptotic limit
holds,
\begin{equation}\label{eq-thm-3}
\lim_{j\to\infty}
\left(j!(\Lambda_1-\ell_j^{(1)})\right)^{1/j}=\frac{b}2\big{(}\rm{Cap}(K)\big{)}^2\,.
\end{equation}\end{cor}
\begin{proof}
Since the function $\gamma\in L^\infty(\partial\Omega)$, then there
exist $\gamma_1<0$ and $\gamma_2>0$ such that
$$|\gamma_1|>C_0,\quad|\gamma_2|>C_0,
\quad{\rm and}~\gamma_1<\gamma(x)<\gamma_2\quad\forall~x\in\partial\Omega\,.$$
Here $C_0=C_0(\Omega)$ is the geometric constant from
Theorem~\ref{thm-KP}.

The variational min-max principle gives that the eigenvalues below
the bottom of the essential spectrum are monotone with respect to
$\gamma$, hence we obtain,
$$\ell^{(1)}_j(\gamma_1)\leq \ell_j^{(1)}(\gamma)\leq
\ell_j^{(1)}(\gamma_2)\,,\quad\forall~j\in\mathbb N\,.$$ Here, for a
boundary function $\eta\in L^\infty(\partial\Omega)$,
$\{\ell_j^{(1)}(\eta)\}_j$ denotes the increasing sequence of
eigenvalues of the operator $L_{\Omega,b}^\eta$ in the interval
$(-\infty,\Lambda_1)$.
Invoking then the asymptotic limit
\eqref{eq-thm-2} and noticing that it is independent from the
boundary condition, we get the asymptotic limit announced in
Corollary~\ref{thm-KP-N}.
\end{proof}

Roughly speaking, similar to \cite{puro}, the idea of the proof of
Theorem~\ref{thm-KP} is to work with the resolvent of
$L_{\Omega,b}^\gamma$, which can be viewed as a compact perturbation of the
resolvent of the Landau Hamiltonian in $\R^2$. To get the asymptotic
accumulation of the eigenvalues, we carry out a reduction to  a
boundary pseudo-differential operator, whose spectrum can be
compared with that of Toeplitz operators.

The paper is organized as follows. In Section~\ref{Sec-prel} we
collect various auxiliary and technical results that will be useful
in the proof. In Section~\ref{Sec-proof}, we give the proof of
Theorem~\ref{thm-KP}.

\section{Preliminaries}\label{Sec-prel}

\subsection{Two abstract results}
In this section we state two abstract results in operator theory
that will be useful in the paper.

\begin{lemma}\label{lem-PR}(Pushnitski-Rozenblum \cite[Proposition~2.1]{puro}).
Assume that $A$ and $B$ are two self-adjoint positive operators in
$L^2(\R^2)$ satisfying the following hypotheses:
\begin{itemize}
\item $0\not\in \sigma(A)\cup\sigma(B)$.
\item The form domain of $A$ contains that of $B$, i.e.
$D(B^{1/2})\subset D(A^{1/2})$.
\item For all $f\in D(B^{1/2})$, $\|A^{1/2}f\|=\|B^{1/2}f\|$, i.e.
the quadratic forms of $A$ and $B$ agree on the form
domain of $B$. \end{itemize} Then, $B^{-1}\leq A^{-1}$ in the quadratic form
sense, i.e.
\begin{displaymath}
\langle B^{-1}f,f\rangle\leq \langle A^{-1}f,f\rangle\quad\forall~f\in L^2(\R^2).
\end{displaymath}
\end{lemma}

The second abstract result we state is Theorem~9.4.7 from
\cite{biso}.

\begin{lemma}\label{lem-biso}
Assume $A$ is a  self-adjoint operator  and $V$ a compact  and
positive operator
in $L^2(\R^2)$ such that the spectrum of $A$   in an interval
$(\alpha,\beta)$ is discrete and does not accumulate at $\beta$.
Then the spectrum of the operator $B=A+V$ in $(\alpha,\beta)$ is
discrete and does not accumulate at $\beta$.
\end{lemma}

\subsection{Some facts about the Landau Hamiltonian}\label{sec:landau}
In this section we review classical results concerning the Landau
Hamiltonian
\begin{equation}\label{eq-LH}
L=-(\nabla-ib\Ab_0)^2\quad\text{in }\mathbb{R}^{2}\,.
\end{equation}
Here $\Ab_0$ is the magnetic
potential with unit constant magnetic field introduced in
\eqref{mp-A0}, and $B$ is a positive constant. The form domain of
$L$ is the magnetic Sobolev space
\begin{displaymath}
H^1_{\Ab_0}(\R^2)=\{u\in L^2(\R^2)~:~(\nabla-ib\Ab_0)u\in
L^2(\R^2)\}\,.
\end{displaymath}
The spectrum of $L$ consists of infinitely degenerate eigenvalues
called Landau levels,
\begin{displaymath}
\sigma(L)=\{\Lambda_n~:~n\in\mathbb
N\}\,,\quad\Lambda_n=(2n-1)b\quad\forall~n\in\mathbb N\,.
\end{displaymath}
We denote by $\mathcal L_n$ the eigenspace associated with the
Landau level $\Lambda_n$, i.e.
\begin{equation}\label{eq-Ln}
\mathcal L_n={\rm Ker}(L-\Lambda_n)\quad\forall~n\in\mathbb
N\,.
\end{equation}
We also denote  by $P_n$ the orthogonal projection on the eigenspace
$\mathcal L_n$.

The operator $L$ can also be expressed by the creation and annihilation
operators $\overline{Q}$ and $Q$. We introduce complex notation $z=x_1+ix_2$ and
let $\Psi=\frac{b}{4}|z|^2$ be a scalar potential, $\Delta\Psi=b$. Then, with
\begin{equation*}
\overline{Q}=-2i e^{-\Psi}\frac{\partial}{\partial z}e^\Psi,\quad
Q = -2i e^\Psi\frac{\partial}{\partial \bar{z}} e^{-\Psi}
\end{equation*}
the following well known identities hold:
\begin{equation}\label{eq-creann}
\overline{Q}=Q^*,\quad
[Q,\overline{Q}]=2b,\quad
L = \overline{Q}Q+b.
\end{equation}
We also notice that we can define $R_0=L^{-1}$, the resolvent of $
L$. This is a bounded operator $R_0\in \mathcal L(L^2(\R^2))$ with
image in $D(L)$. Furthermore, $R_0$ is an operator with an integral
kernel that we denote by $G_0(x,y)$. The following formula for $G_0(x,y)$ is
given in \cite{avhesi} (here $x\wedge y=x_1y_2-x_2y_1$):
\begin{equation}\label{eq-kernel}
G_0(x,y)=\int_0^\infty \frac{b}{4\pi\sinh(bs)}\exp\Bigl(\frac{ib}{2}x\wedge y
-\frac{b}{4\tanh(bs)}|x-y|^2 \Bigr)\ed s.
\end{equation}
.
\begin{lemma}
\label{lem:kernel}
 $R_0$ is an integral operator with kernel $G_0(x,y)$ that has the
 following singularity at the diagonal,
\begin{equation}
 \label{eq:as}
G_0(x,y) \sim \frac{1}{2\pi}
\ln\left(\frac1{|x-y|}\right)+O(1)\quad \text{as } |x-y|\to 0,
\end{equation}
and the corresponding behavior holds for $\partial_N G_0(x,y)$.
Moreover, $G_0(x,y)$ decays as a Gaussian as $|x-y|\to\infty$.
\end{lemma}

\begin{proof}
This integral can be expressed in terms of
the Whittaker function
(see~\cite[Section~4.9, formula~(31)]{baer1} and~\cite[Chapter~6]{baer2}) as
\begin{multline*}
G_0(x,y)=\frac{\pi^{3/2}}{b}\Bigl(\frac{b}{8}|x-y|^2\Bigr)^{-3/4}
\exp\Bigl(\frac{ib}{2}x\wedge y\Bigr)\\
\times\Bigl[W_{\frac12,-\frac12}\Bigl(\frac{b}{2}|x-y|^2\Bigr)
+\frac12W_{-\frac12,-\frac12}\Bigl(\frac{b}{2}|x-y|^2\Bigr)\Bigr].
\end{multline*}
The result follows from asymptotic formulae for Whittaker
functions~\cite{baer2}.
\end{proof}

\subsection{Some boundary operators}
Recall that $K\subset\R^2$ has been assumed to be a compact simply
connected subset of $\R^2$ and that we defined $\Omega=\R^2\setminus
K$. Since $\Omega$ and $K$ are complementary, the Hilbert space
$L^2(\R^2)$ is decomposed as the direct sum $L^2(\Omega)\oplus
L^2(K)$ in the sense that any function $u\in L^2(\R^2)$ can be
represented as $u_\Omega\oplus u_K$ where $u_\Omega$ and $u_K$ are
the restrictions of $u$ to $\Omega$ and $K$ respectively.

Denoting by $\Gamma$ the common boundary of $\Omega$ and $K$, we
define the following operator on $\Gamma$,
\begin{equation}\label{eq-normalder}
\partial_\Gamma u=\partial_Nu+\gamma\,u=\nu_\Omega\cdot(\nabla-ib\Ab_0)u+\gamma\,u\,,
\end{equation}
where $\nu_\Omega$ is the unit {\it outward} normal vector to the
boundary of $\Omega$ and $\Ab_0$ is the magnetic potential from
\eqref{mp-A0}. The operator $\partial_\Gamma$ acts on functions in
$H^1_{\text{loc}}(\Omega)$ or in $H^1(K)$. We may write
$(\partial_\Gamma)_x$ in order to stress that the differentiation in
\eqref{eq-normalder} is with respect to the variable $x$.

With $G_0(x,y)$ as in~\eqref{eq-kernel} we define the operators $\mathcal{A}$,
$\mathcal{B}$, $A$ and $B$ as
\begin{equation}\label{eq-boundaryop}
\begin{aligned}
\mathcal{A}\alpha(x) &= \int_\Gamma
G_0(x,y)\alpha(y)\ed S(y),\quad x\in\R^2\\
\mathcal{B}\alpha(x) &= \int_\Gamma
({\partial_N})_y G_0(x,y)\alpha(y)\ed S(y),\quad x\in\R^2\setminus\Gamma,\\
A\alpha(x) &= \int_\Gamma
G_0(x,y)\alpha(y)\ed S(y),\quad x\in\Gamma\\
B\alpha(x) &= \int_\Gamma
({\partial_N})_y G_0(x,y)\alpha(y)\ed S(y),\quad x\in\Gamma.
\end{aligned}
\end{equation}

\begin{remark}\label{rem:kernel}
The operators $A$ and $B$ in~\eqref{eq-boundaryop} are well-defined
bounded operators in $L^2(\Gamma)$. This is due to the behavior of
the integral kernel $G_0$ from Lemma~\ref{lem:kernel}. Actually, for
a fixed $x\in\R^2$, the function $G_0(x-\cdot)\in L^2(\Gamma)$.

Moreover, since $G_0$ and $\partial_N G_0$ are in
$L^2(\Gamma\times\Gamma)$, we see that the operators $A$ and $B$
are Hilbert-Schmidt, hence compact in $L^2(\Gamma)$.\hfill$\Box$
\end{remark}

\begin{lemma}\label{lem-green}
Let $u\in L^2(\R^2)$ be such that $u_\Omega\in H^1_{\text{loc}}(\Omega)$
and $u_K\in H^1(K)$. Then it holds that
\begin{align*}
\left(B+\left(\gamma+\frac12\right)I\right)u_\Omega&=A(\partial_\Gamma
u_\Omega)\quad\text{and}\\
\left(B+\left(\gamma-\frac12\right)I\right)u_K&=A(\partial_\Gamma
u_K).
\end{align*}
\end{lemma}
\begin{proof}
It is proved in \cite[(4.6)-(4.7)]{P} that,
\begin{displaymath}
\left(B+\frac12I\right)u_\Omega=A(\partial_N
u_\Omega)\quad\text{and}\quad \left(B-\frac12I\right)u_K=A(\partial_N
u_K)\,.
\end{displaymath}
Inserting $\partial_N=\partial_\Gamma-\gamma I$ above, we get the
formulae that we wish to prove.
\end{proof}

\begin{lemma}\label{lem:pseudo}
The operators $A$ and $B$ are pseudo-differential operators of order $-1$.
Moreover, the operator $A$ is elliptic and
$A:L^2(\Gamma)\to H^1(\Gamma)$ is
an isomorphism.
\end{lemma}
\begin{proof}
That $A$ and $B$ are pseudo-differential operators of order
$-1$ and the fact that $A$ is elliptic is due to the asymptotic
behavior of the Green's potential $G_0$
near the diagonal (see Lemma~\ref{lem:kernel}. We refer the reader to
Taylor's book \cite[Chapter~7, Proposition~11.2]{tay2} for a proof.

To prove that $A$ is an isomorphism, we follow the proof of \cite[Chapter~7,
Proposition~11.5]{tay2} with the necessary modifications.

Assume that $h\in C^\infty(\Gamma)$ with $Ah=0$. If we define $u(x)$ by
$u(x)=\mathcal{A}h(x)$, $x\in K^\circ$, then $u$ satisfies
\begin{equation*}
\begin{cases}
-(\nabla-ib\mathbf{A}_0)^2u = 0 & \text{in $K^\circ$} \\
u = 0 & \text{on $\Gamma$.}
\end{cases}
\end{equation*}
We use~\eqref{eq-creann} and integrate by parts, to get
\begin{equation*}
0 = \langle -(\nabla-b\mathbf{A}_0A)^2 u,u\rangle_{L^2(K)} = b\|u\|_{L^2(K)}^2
+ \|Qu\|_{L^2(K)}^2
\end{equation*}
and so $u\equiv 0$ in $K$, i.e.
\begin{equation}\label{eq-Ah0}
\mathcal{A}h(x)\equiv 0\quad \text{in $K^\circ$.}
\end{equation}

It follows from~\cite[Chapter~7, Proposition~11.3]{tay2} that
$\partial_N(\mathcal{A}h)(x)$ makes
a jump across the boundary $\Gamma$ of size $h$, so if we let
$w(x)=\mathcal{A}h(x)$, $x\in\Omega$, then it satisfies
\begin{equation}\label{eq-outer}
\begin{cases}
-(\nabla-ib\mathbf{A}_0)^2w = 0 & \text{in $\Omega$} \\
\partial_N w = h & \text{on $\Gamma$.}
\end{cases}
\end{equation}
Since $\mathcal{A}h$ does not jump across $\Gamma$, we see by~\eqref{eq-Ah0}
that $w=0$ on $\Gamma$.

From the exponential decay of $G_0(x,y)$ as $|x-y|\to\infty$ it follows that
$w(x)=O(|x|^{-N})$ as $|x|\to\infty$ for all $N>0$.
Moreover $w$ is smooth. Hence we can integrate by parts in $\Omega$ to find
\begin{equation*}
0 = \langle -(\nabla-ib\mathbf{A}_0)^2 w,w\rangle_{L^2(\Omega)} = b\|w\|_{L^2(\Omega)}^2
+ \|Qw\|_{L^2(\Omega)}^2,
\end{equation*}
and hence $w\equiv 0$ in $\Omega$. From~\eqref{eq-outer} we see that $h=0$.
\end{proof}

We conclude the section with the following lemma.

\begin{lemma}\label{lem:RobDirch}
There exists a positive constant $C_0$ depending only on $\Gamma$
such that, assuming $\gamma\in C^\infty(\Gamma)$  satisfies
$|\gamma(x)|>C_0$ for all $x\in\Gamma$, then for any function
$u=u_\Omega\oplus u_K\in L^2(\R^2)$ satisfying $u_\Omega\in H^1_{\rm
loc}(\Omega)$ and $u_K\in H^1(K)$, it holds that,
\begin{displaymath}
\partial_\Gamma
u_\Omega=A^{-1}\left(B+\left(\gamma+\frac12\right)I\right)u_\Omega\,,\quad
\partial_\Gamma
u_K=A^{-1}\left(B+\left(\gamma-\frac12\right)I\right)u_K\,,
\end{displaymath}
\begin{displaymath}
u_\Omega=\left(B+\left(\gamma+\frac12\right)I\right)^{-1}A(\partial_\Gamma
u_\Omega)\,, \quad
u_K=\left(B+\left(\gamma-\frac12\right)I\right)^{-1}A(\partial_\Gamma
u_K)\,.
\end{displaymath}
Here $A$ and $B$ are the operators from
\eqref{eq-boundaryop}.
\end{lemma}
\begin{proof}
Pick $C_0$ such that $C_0>\|B\|+1$. Since $B$ is a bounded operator,
the spectrum of $B$ is strictly included in the open ball of center
$0$ and radius $C_0$. The hypothesis we made on $\gamma$ guarantees
that
\begin{displaymath}
0\not\in\sigma \left(B+\left(\gamma\pm\frac12\right)I\right)\,.
\end{displaymath} Thus we can invert the operator
$B+\left(\gamma\pm\frac12\right)I$. Now  invoking
Lemma~\ref{lem-green} finishes the proof of the lemma.
\end{proof}

\subsection{A result on Toeplitz operators}
Recall the Landau levels $\{\Lambda_n\}_{n\in\mathbb N}$ together
with their eigenspaces $\{\mathcal L_n\}_{n\in\mathbb N}$ introduced in
Section~\ref{sec:landau}. For all $n\in\mathbb N$, we denoted by
$P_n$ the orthogonal projector on the space $\mathcal L_n$. Given a
positive integer $n\in\mathbb N$ and a compact simply connected
domain $U\subset\R^2$ with smooth boundary, the Toeplitz operator
$S_n^U$ is defined by,
\begin{equation}\label{eq-toep}
S_n^U=P_n\chi_UP_n\quad \text{in }L^2(\R^2)\,.
\end{equation}
Here $\chi_U$ is the characteristic function of $U$. Since
$\im(\chi_U P_n)\subset H^1(U)$ and the boundary of $U$ is smooth,
then $\chi_U P_n$ is a compact operator, and so is the Toeplitz
operator $S_n^U$.

We state the following lemma which we take from
\cite[Lemma~3.2]{fipu}.

\begin{lemma}\label{lem:toepspectrum}
Given $n\in\mathbb N$, denote by $s_1^{(n)}\geq s_2^{(n)}\geq
\ldots$ the decreasing sequence of eigenvalues of $S^U_n$. Then,
\begin{equation*}
\lim_{j\to\infty}\big(j! s_j^{(n)}\big)^{1/j} =
\frac{b}{2}\left(\kap(U)\right)^2.
\end{equation*}
\end{lemma}

\section{Proof of Theorem~\ref{thm-KP}}\label{Sec-proof}
Recall the compact simply connected smooth domain $K\subset\R^2$ and
the exterior domain $\Omega=\R^2\setminus K$. We have introduced the
operator $L_{\Omega,B}^\gamma$ with quadratic form
$q_{\Omega,B}^\gamma$ from \eqref{QF-q}. We will use also the
corresponding operator in $K$, namely $L_{K,B}^{-\gamma}$.

Since the quadratic forms $q_{\Omega,b}^\gamma$ and
$q_{K,b}^{-\gamma}$ are semi-bounded, we get up to a shift by a
positive constat that they are strictly positive. Thus we assume,
the hypothesis:
\begin{itemize}
\item[(H1)] The quadratic forms $q_{\Omega,b}^\gamma$ and
$q_{K,b}^{-\gamma}$ from \eqref{QF-q} are strictly positive.
\end{itemize}

The relevance of the hypothesis (H1) is that it provides us with the
existence of the resolvents of $L_{\Omega,b}^\gamma$ and
$L_{K,b}^{-\gamma}$.

When there is no ambiguity, we will skip $b$ and $\gamma$ from the
notation, and write $L_\Omega$, $L_K$, $q_\Omega$ and $q_K$ for the
operators $L_{\Omega,b}^\gamma$, $L_{K,b}^{-\gamma}$, the quadratic
forms $q_{\Omega,b}^\gamma$ and $q_{K,B}^{-\gamma}$ respectively.
Notice that, for all $u=u_\Omega\oplus u_K\in L^2(\R^2)$ such that
$u_\Omega\in D(q_\Omega)$ and $u_K\in D(q_K)$, we have,
\begin{align*}
q_\Omega(u_\Omega)&=\int_\Omega|(\nabla-ib\Ab_0)u_\Omega|^2\ed x+\int_\Gamma\gamma|u_\Omega|^2\ed S\\
q_K(u_K)&=\int_K|(\nabla-ib\Ab_0)u_K|^2\ed x-\int_\Gamma\gamma|u_K|^2\ed S\,.
\end{align*}
If in addition, $u\in H^1_{\Ab_0}(\R^2)$, then
$q_\Omega(u_\Omega)+q_K(u_K)=\int_{\R^2}|(\nabla-ib\Ab_0)u|^2\ed x$.
We point also that if $u_\Omega\in D(L_\Omega)$ and $u_K\in D(L_K)$,
then
\begin{displaymath}
\partial_\Gamma u_\Omega=\partial_\Gamma u_K=0,
\end{displaymath}
where $\partial_\Gamma$ is the trace operator from
\eqref{eq-normalder}.

\subsection{Extension of $L_\Omega$ to an operator in $L^2(\R^2)$}
We pointed earlier that since $\Omega$ and $K$ are complementary in
$\R^2$, the space $L^2(\R^2)$ is decomposed as a direct sum
$L^2(\Omega)\oplus L^2(K)$. This permits us to extend the operator
$L_\Omega$ in $L^2(\Omega)$ to an operator $\widetilde L$ in
$L^2(\R^2)$. Actually, let $\widetilde L=L_\Omega\oplus L_K$ in
$D(L_\Omega)\oplus D(L_K)\subset L^2(\R^2)$. More precisely,
$\widetilde L$ is the self-adjoint extension associated with the
quadratic form
\begin{equation}\label{eq-QF-q}
\widetilde q(u)=q_\Omega(u_\Omega)+q_K(u_K)\,,\quad u=u_\Omega\oplus
u_K\in L^2(\R^2)\,.
\end{equation}
By our hypothesis (H1), we may speak of the resolvent $\widetilde
R=\widetilde L^{-1}$ of $\widetilde L$. We then have the following
lemma.
\begin{lemma}\label{lem:esssp}
With $\widetilde L$, $\widetilde R$ and $L_\Omega$ defined as above,
it holds that:
\begin{enumerate}
\item $\sigma_{\text{ess}}(L_\Omega)=\sigma_{\text{ess}}(\widetilde L)$.
\item $\lambda\in\sigma_{\text{ess}}(\widetilde R)\setminus\{0\}$ if
and only if $\lambda\not=0$ and  $\lambda^{-1}\in\sigma_{\text{ess}}(L_\Omega)$.
\end{enumerate}
\end{lemma}
\begin{proof}
Since $\widetilde L=L_\Omega\oplus L_K$, then $\sigma(\widetilde
L)=\sigma(L_\Omega)\cup\sigma(L_K)$. But $K$ is compact and has a
smooth boundary, hence $L_K$  has a compact resolvent. Thus
$\sigma_{\text{ess}}(L_K)=\varnothing$ and here it follows the first
assertion in the lemma above. Moreover, $L_\Omega$ and $L_K$ are
both strictly positive by hypothesis, hence $0\not\in\sigma(
\widetilde L)$. It is then straight forward that
\begin{displaymath}
\sigma_{\text{ess}}(\widetilde
L)=\{\lambda\in\R\setminus\{0\}~:~\lambda^{-1}
\in\sigma_{\text{ess}}(\widetilde R)\}.
\end{displaymath}
\end{proof}

\subsection{Essential spectrum of $L_\Omega$} With the operator
$\widetilde L$ introduced above, we can view $L_\Omega$ as a
perturbation of the Landau Hamiltonian $L$ in $\R^2$ introduced in
\eqref{eq-LH}. Actually, we define
\begin{displaymath}
V=\widetilde R-R_0=\widetilde L^{-1}-L^{-1}.
\end{displaymath}
Then we have the following result on the operator $V$.

\begin{lemma}\label{lem:V}
The operator $V\in\mathcal L(L^2(\R^2))$ is positive and compact.
Moreover, for all $f,g\in L^2(\R^2)$, it holds that
\begin{equation}\label{eq-V}
\langle f,Vg\rangle=\int_{\Gamma}\partial_\Gamma
u\cdot\overline{(v_\Omega-v_K)}\ed S\,,\end{equation} where $u=R_0f$
and $v=\widetilde R g$.
\end{lemma}
\begin{proof}
Notice that the form domain $H_{\Ab_0}^1(\R^2)$ of $L$ is included
in that of $\widetilde L$, and that for $u\in H^1_{\Ab_0}(\R^2)$, we
have
\begin{displaymath}
\widetilde q(u)=\int_{\R^2}|(\nabla-ib\Ab_0)u|^2\ed x\,.
\end{displaymath}
Invoking Lemma~\ref{lem-PR}, we get that the operator $V$ is
positive.

Let us establish the identity in \eqref{eq-V}. Notice that $f=Lu$
and $g=\widetilde L v= L_\Omega v_\Omega\oplus L_K v_K$. Then we
have,
\begin{displaymath}
\langle f,Vg\rangle=\int_\Omega Lu\cdot
\overline{v_\Omega}\ed x+\int_KLu\cdot\overline{v_K}\ed x-\int_\Omega
u\cdot\overline{L_\Omega v_\Omega}\ed x-\int_K
u\cdot\overline{L_Kv_K}\ed x\,.
\end{displaymath}
The identity in \eqref{eq-V} then
follows by integration by parts and by using the boundary conditions
$\partial_\Gamma v_\Omega=\partial_\Gamma v_K=0$.

Knowing that the trace operators are compact, we conclude from
\eqref{eq-V} that $V$ is a compact operator.
\end{proof}

As corollary of Lemma~\ref{lem:V}, we get the first part of
Theorem~\ref{thm-KP} proved.

\begin{cor}\label{corl:proof-mainthm}
Assume the hypothesis (H1) above holds. Then
\begin{displaymath}
\sigma_{\text{ess}}(L_\Omega)=\{\Lambda_n~:~n\in\mathbb N\}\,,\quad
\Lambda_n=(2n-1)b\,,
\end{displaymath}
and for all $\varepsilon\in(0,b)$,
\begin{displaymath}
\sigma(L_\Omega)\cap(\Lambda_n,\Lambda_n+\varepsilon)\quad\text{is finite}
\quad\forall ~n\in\mathbb N\,.
\end{displaymath}
\end{cor}
\begin{proof}
Invoking Lemma~\ref{lem:esssp}, it suffices to prove that
$\sigma_{\text{ess}}(\widetilde R)=\{\Lambda_n^{-1}~:~n\in\mathbb N\}$
in order to get the result concerning the essential spectrum of
$L_\Omega$. Notice that $\widetilde R=R_0+V$ with $V$ a compact
operator. Hence by Weyl's theorem, $\sigma_{\text{ess}}(\widetilde
R)=\sigma_{\text{ess}}(R_0)=\sigma_{\text{ess}}(L^{-1})$. But we know from
Section~\ref{sec:landau} that $\sigma_{\text{ess}}(R_0)
=\{\Lambda_n^{-1}~:~n\in\mathbb N\}$ as was required to
prove.

Since the operator $V$ is compact and positive, invoking Lemma~\ref{lem-biso}, we
get that $\sigma(\widetilde
R)\cap(\Lambda_n^{-1}-\varepsilon,\Lambda_n^{-1})$ is finite. Since
$\widetilde R=\widetilde L^{-1}$, this gives that
$\sigma(L_\Omega)\cap(\Lambda_n,\Lambda_n+\varepsilon)$ is finite.
\end{proof}

\subsection{Reduction to a Toeplitz operator}

In light of Corollary~\ref{corl:proof-mainthm}, we have only to
establish the second part of Theorem~\ref{thm-KP}, namely the
asymptotic formulae in \eqref{eq-thm-2}.

Let $n\in\mathbb N$ and pick $\tau>0$ such that
$\big((\LL_n^{-1}-2\tau,\LL_n^{-1}+2\tau)\setminus\{\LL_n^{-1}\}\big)\cap
\sigma_{\text{ess}}(\widetilde R)=\varnothing$. Denote by
$\{r_j^{(n)}\}_{j\geq 1}$ the decreasing sequence of eigenvalues of
$\widetilde R$ in the interval
$(\Lambda^{-1}_n,\Lambda_n^{-1}+\tau)$. In order to prove
\eqref{eq-thm-2}, it suffices to show that
\begin{equation}\label{eq-thm-resolvent}
\lim_{j\to\infty}\left(j!\left(r_j^{(n)}-\Lambda_n^{-1}\right)\right)^{1/j}
=\frac{b}2\left(\text{Cap}(K)\right)^2\,.
\end{equation}

We introduce the operator \begin{equation}\label{eq-Tn}
T_n=P_nVP_n\,, \end{equation} where $P_n$ is the orthogonal
projection on the eigenspace $\mathcal L_n$ associated with
$\Lambda_n$. By Lemma~\ref{lem:V}, $V$ is  a compact operator, hence
$T_n$ is also a compact operator.  Denote by $\{t_j^{(n)}\}$ the
decreasing sequence of eigenvalues of $T_n$.

The next lemma, proved in \cite[Proposition~2.2]{puro}, shows that
$r_j^{(n)}-\Lambda_n^{-1}$ are close to the eigenvalues of $T_n$.
\begin{lemma}\label{lem:RT}
Given $\varepsilon>0$ there exist integers $l$ and $j_0$ such that
\begin{equation*}
 (1-\varepsilon)t_{j+l}^{(n)}\leq r_{j}^{(n)}-\LL_n^{-1} \leq (1+\varepsilon)t_{j-l}^{(n)}, \quad\forall~j\geq j_0.
\end{equation*}
\end{lemma}

In all what follows, we work under the following additional
hypothesis:
\begin{itemize}
\item[(H2)] The function $\gamma\in L^\infty(\Gamma)$ satisfies
\begin{equation}\label{eq-hyp-gam}
\min_{x\in\Gamma}|\gamma(x)|>C_0\,,
\end{equation}
where $C_0$ is the geometric constant from Lemma~\ref{lem:RobDirch}.
\end{itemize}

Under the hypothesis (H2) above, the spectrum of $T_n$ will be
further related to the spectra of Toeplitz operators. Recall that
associated with a compact domain $U\subset\R^2$, we introduced in
\eqref{eq-toep} the Toeplitz operator $S_n^U$. We will prove the
following result.

\begin{lemma}\label{lem:toeplitztva}
Let $K_0\subset K \subset K_1$ be compact domains with ${\partial
K_i\cap \partial K=\varnothing}$. There exists a constant $C>1$ such
that
\begin{equation}\label{eq:inequality}
 \frac1C \langle f,S_n^{K_0}f\rangle \leq \langle f, T_n f \rangle \leq C \langle
 f,S_n^{K_1}f\rangle\quad\forall~f\in L^2(\R^2)\,.
\end{equation}
\end{lemma}

The proof of Lemma~\ref{lem:toeplitztva} is by reduction of the
operator $T_n$ to a pseudo-differential operator on the common
boundary $\Gamma$ of $\Omega$ and $K$. We will give the proof in the
next section, but we give first the proof of
\eqref{eq-thm-resolvent}.
\begin{cor}\label{cor:thm-resolvent}
Assume the hypotheses (H1) and (H2) above hold. Then the claim in
\eqref{eq-thm-resolvent} above is true.
\end{cor}
\begin{proof}
Invoking the variational min-max principle, the result of
Lemma~\ref{lem:toeplitztva} provides us with a sufficiently large
integer $j_0\in\mathbb N$ such that, for all $j\geq j_0$, we have,
\begin{displaymath}
\frac1C s_{j,K_0}^{(n)}\leq t_j^{(n)}\leq C s_{j,K_1}^{(n)}\,.
\end{displaymath}
Here $\{s_{j,K_0}^{(n)}\}_j$ and $\{s_{j,K_1}^{(n)}\}_j$ are the
decreasing sequences of $S_n^{K_0}$ and $S_n^{K_1}$ respectively.
Implementing the result of Lemma~\ref{lem:toepspectrum} in the
inequality above, we get
\begin{displaymath}
\frac{b}2\left(\text{Cap}(K_0)\right)^2
\leq\lim_{j\to\infty}\left(j!t_j^{(n)}\right)^{1/j}
\leq\frac{b}2\left(\text{Cap}(K_1)\right)^2.
\end{displaymath}
Since both $K_0\subset
K\subset K_1$ are arbitrary, we get by making them close to $K$,
\begin{displaymath}
\lim_{j\to\infty}\left(j!t_j^{(n)}\right)^{1/j}=\frac{b}2
\left(\text{Cap}(K)\right)^2\,.
\end{displaymath}
Implementing the above asymptotic limit in the
estimate of Lemma~\ref{lem:RT}, we get the announced result in
Corollary~\ref{cor:thm-resolvent} above.
\end{proof}

Summing up the results of Corollaries~\ref{corl:proof-mainthm} and
\ref{cor:thm-resolvent}, we end up with the proof of
Theorem~\ref{thm-KP} provided the hypotheses (H1) and (H2) are
verified. As we explained earlier, the hypothesis (H1) can be
eliminated by shifting the operator by a sufficiently large positive
constant. On the other hand, Theorem~\ref{thm-KP} is stated under
the hypothesis (H2) on the function $\gamma$.

So all what we need now is to prove Lemma~\ref{lem:toeplitztva}.
That will be the subject of the next section.

\subsection{Reduction to a boundary pseudo-differential operator}

We start with the following reduction of the operator $T_n$ from
\eqref{eq-Tn}.

\begin{lemma}\label{lem:Tn}
Under the hypothesis (H2) above, for all $f,g\in L^2(\R^2)$, it
holds that,
\begin{equation}\label{eq-lem:Tn}
\langle f,T_n g\rangle=\frac{1}{\Lambda_n^2}\int_\Gamma
(P_nf)\cdot\overline{T((P_ng))}\ed S\,.
\end{equation}
Here,
\begin{gather*}
T=T_{B,-}A^{-1}T_{B,+}^{-1}\,,\\
T_{B,\pm}=
\left(B+\left(\gamma\pm\frac12\right)I\right)\,,
\end{gather*}
and $A$, $B$ the
operators from \eqref{eq-boundaryop}.
\end{lemma}
\begin{proof}
Recall that $R_0=L^{-1}$ is the resolvent of the Landau Hamiltonian,
$\widetilde R_0$ that of the Hamiltonian $\widetilde
L=L_\Omega\oplus L_K$. We denote by $u=R_0P_nf=\Lambda_n^{-1}P_nf$,
$v=\widetilde R P_ng=v_\Omega\oplus v_K$ and
$w=R_0P_ng=\Lambda_n^{-1}P_ng$. Notice that
\begin{displaymath}
\langle f,T_ng\rangle=\langle P_nf,VP_ng\rangle
\end{displaymath}
where $V$ is the operator from \eqref{eq-V}. Invoking
Lemma~\ref{lem:V}, we write,
\begin{align*}
\langle P_nf,VP_ng\rangle &= \int_\Gamma\partial_\Gamma
u\cdot\overline{(v_\Omega-v_K)}\ed S\\
&=\int_\Gamma\partial_\Gamma
u\cdot\overline{(v_\Omega-w+w-v_K)}\ed S
\end{align*}
Using Lemma~\ref{lem:RobDirch}, we can write further,
\begin{displaymath}
\langle P_nf,VP_ng\rangle=\int_\Gamma\partial_\Gamma
u\cdot\overline{(T_{B,-}^{-1}A(\partial_\Gamma
(v_\Omega-w))+T_{B,+}^{-1}A(\partial\Gamma(w-v_K)))}\ed S.
\end{displaymath}
Notice
that $v_\Omega$ and $v_K$ are in the domain of the operators
$L_\Omega$ and $L_K$ respectively, hence $\partial_\Gamma
v_\Omega=\partial_\Gamma v_K=0$. Consequently,
\begin{equation}\label{eq-lem-RD}\langle P_nf,VP_ng\rangle=
\int_\Gamma\partial_\Gamma
u\cdot\overline{(T_{B,+}^{-1}-T_{B,+}^{-1})A(\partial_\Gamma
w))}\ed S.\end{equation} Lemma~\ref{lem:RobDirch} also gives,
\begin{displaymath}
\partial_\Gamma
w=A^{-1}T_{B,-}w\,,\quad \partial_\Gamma u=A^{-1}T_{B,-}u\,.
\end{displaymath}
Inserting this in \eqref{eq-lem-RD}, we get by a simple calculation,
\begin{displaymath}
\langle P_nf,VP_ng\rangle=
\int_\Gamma\partial_\Gamma u\cdot\overline{Tw}\ed S\,,
\end{displaymath}
where $T$ is the operator introduced in Lemma~\ref{lem:Tn} above.
Recalling that $u=R_0P_nf=\Lambda_n^{-1}P_nf$ and
$w=R_0P_ng=\Lambda_n^{-1}P_ng$, we get the identity announced in
Lemma~\ref{lem:Tn} above.
\end{proof}

\begin{lemma}\label{lem:Tn2}
There exists a constant $C>1$, for all $f\in L^2(\R^2)$ it holds
that,
\begin{displaymath}
\frac1C\|P_nf\|_{L^2(\Gamma)}\|P_nf\|_{H^1(\Gamma)}\leq \langle
f,T_nf\rangle\leq C\|P_nf\|_{L^2(\Gamma)}\|P_nf\|_{H^1(\Gamma)}\,.
\end{displaymath}
Here $P_n$ is the orthogonal projection on the Landa level $\mathcal L_n$.
\end{lemma}
\begin{proof}
Lemma~\ref{lem:pseudo} says that $A$ is an elliptic
pseudo-differential operator of order $-1$, hence $A^{-1}$ is a
pseudo-differential operator of order $1$. On the other hand,
$B$ is also a pseudo-differential operator of order $-1$ and
$\gamma\in C^\infty(\Gamma)$, hence
$T_{B,\pm}$ and $T_{B,\pm}^{-1}$ are pseudo-differential operators
with order $0$. Therefore, the operator $T$ from Lemma~\ref{lem:Tn} is
a pseudo-differential operator with order $1$. Invoking again
Lemma~\ref{lem:pseudo}, $T$ is invertible and therefore, there exists
a constant $C>1$ such that,
$$\frac1C\|u\|_{L^2(\Gamma)}\|u\|_{H^1(\Gamma)}\leq \langle
u,Tu\rangle\leq C\|u\|_{L^2(\Gamma)}\|u\|_{H^1(\Gamma)}\,,$$
for all $f\in H^1(\Gamma)$. Applying the above estimate with $u=P_nf$
and $f\in L^2(\R^2)$, and recalling \eqref{eq-lem:Tn}, we get
the double inequality announced in
the above lemma holds for all $f\in\mathcal S$.
\end{proof}

\begin{proof}[Proof of Lemma~\ref{lem:toeplitztva}]\ \\
{\it Step~1. Lower bound.} We prove that  the lower bound in
\eqref{eq:inequality} is valid for all $f\in L^2(\R^2)$.  Let
$u=P_nf$, with $P_n$ the orthogonal projection on the eigenspace
$\mathcal L_n$ associated with the Landau level $\Lambda_n$.  By the
definition of $T_n$ from \eqref{eq-Tn}, the estimate of
Lemma~\ref{lem:Tn2} gives,
\begin{displaymath}
\langle f,T_nf\rangle\geq
\frac1C\|u\|_{L^2(\Gamma)}\|u\|_{H^1(\Gamma)}\,.
\end{displaymath}
So it suffices to
prove that
\begin{displaymath}
\langle f,S_n^{K_0}f\rangle\leq
C'\|u\|_{L^2(\Gamma)}\|u\|_{H^1(\Gamma)}\,,
\end{displaymath}
for some positive constant $C'$. Recalling the
definition of $S_n^{K_0}$, this is equivalent to showing that
\begin{equation}\label{eq-lem-ppp}
\|u\|_{L^2(K_0)}\leq C'\|u\|_{L^2(\Gamma)}\|u\|_{H^1(\Gamma)}\,.
\end{equation}
Notice that $L_nu=0$, where $L_n=L-\Lambda_nI$. Let us denote by
$E(x,y)$ the Green's potential of the operator $L_n$. Then $E$ is
smooth away from the diagonal $x=y$ and decays logarithmically
near the diagonal in the same way described in
Lemma~\ref{lem:kernel} (see Stampacchia \cite{Stamp}).

Let $B_n$ be the double layer operator evaluated at the boundary,
i.e. for any $\alpha\in H^1_{\text{loc}}(\R^2)$,
\begin{equation*}
 B_n\alpha(x) = \int_\Gamma (\partial_N)_y E(x,y)\alpha(y)\,d S(y),\quad x\in\Gamma.
\end{equation*}
Here we remind the reader that
$\partial_N=\nu_\Omega\cdot(\nabla-ib\Ab_0)$ and $\nu_\Omega$ is
unit outward normal of the boundary $\partial\Omega=\Gamma$.

The operator $B_n$ is compact, since the kernel $ E(x,y)$ has a
logarithmic singularity at the diagonal $x=y$. Therefore, under the
hypothesis that $|\gamma|$ is sufficiently large, we can invert the
operator $B_n+(\gamma+\frac12)I$ in $L^2(\R^2)$. That way, similar
to Lemma~\ref{lem-green}, using the results in \cite[Chapter~7,
Section~11]{tay2}, we can write,
\begin{equation}\label{eq:fhelp}
u(x)=\int_{\Gamma}(\partial_N)_y
E(x,y)\left(B_n+\left(\gamma+\frac12\right)I\right)^{-1}u(y)\,d
S(y),
 \quad x\in K^\circ\,,
\end{equation}
for all $u\in P_n(L^2(\R^2))$.  Thus \eqref{eq:fhelp} gives us the
inequality $\|u\|_{L^2(K_0)} \leq C' \|u\|_{L^2(\Gamma)}$, which is
 sufficient to deduce the estimate in \eqref{eq-lem-ppp}
above.

\noindent{\it Step~2. Upper bound.}\\
Now we establish the upper bound in \eqref{eq:inequality}. This part
is actually quiet easy.  Let $f\in L^2(\R^2)$ and $u=P_nf$, the
projection on the eigenspace $\mathcal L_n$. Notice that the trace
theorem gives,
\begin{displaymath}
\|u\|_{L_2(\Gamma)}\|u\|_{H^1(\Gamma)}\leq C
\|u\|_{H^{1/2}(K)}\|u\|_{H^{3/2}(K)}\,,
\end{displaymath}
for some positive constant
$C$. Invoking the Sobolev-Rellich embedding theorem, we get for a
possibly new constant $C$,
\begin{equation*}
\|u\|_{L_2(\Gamma)}\|u\|_{H^1(\Gamma)}\leq
 C
\|u\|_{H^2(K)}^2\,.
\end{equation*}
Notice that $L_nu=Lu-\Lambda_nu=0$. Then by elliptic regularity,
given a domain $K_1$ such that $K\subset K_1$, there exists a
constant $C_{K_1}$ such that,
\begin{displaymath}
\|u\|_{H_2(\Gamma)}
\leq C_{K_1}
\left(\|L_nu\|_{L^2(K_1)}+\|u\|_{L_2(K_1)}\right)=C_{K_1}\|u\|_{L^2(K_1)}.
\end{displaymath}
Summing up, we get,
\begin{displaymath}
\|P_nf\|_{L_2(\Gamma)}\|P_nf\|_{H^1(\Gamma)}\leq
C\|P_nf\|_{L^2(K_1)}^2\,,\quad\forall~f\in L^2(\R^2)\,.
\end{displaymath}
Substituting the above inequality in the estimate of
Lemma~\ref{lem:Tn2}, we get the upper bound announced in
\eqref{eq:inequality}.
\end{proof}

\section*{Acknowledgements}
AK is supported by a Starting Independent Researcher
  grant by the ERC under the FP7. MP is  supported by the Lundbeck Foundation.
Part of this work has been prepared in ESI - Vienna which is gratefully acknowledged.

\end{document}